\documentclass{amsart}[11pt]

\usepackage{color}
\usepackage{graphicx}
\usepackage{amsmath,amsthm,amscd,amsfonts,amssymb,longtable,mathrsfs, float}
\usepackage[margin=1in]{geometry}
\usepackage{tikz}
\usetikzlibrary{patterns}
\theoremstyle{plain}
\newtheorem{theorem}{Theorem}
\newtheorem*{theorem*}{Theorem}
\newtheorem{lemma}[theorem]{Lemma}
\newtheorem*{lemma*}{Lemma}

\theoremstyle{definition}

\newtheorem*{definition*}{Definition}

\DeclareMathOperator{\Z}{\mathbb{Z}}
\DeclareMathOperator{\im}{\leq_{im}}
\DeclareMathOperator{\nim}{\nleq_{im}}

\DeclareMathOperator{\m}{\leq_{m}}

\title{A Note On Immersion Intertwines of Infinite Graphs}
\author{Matthew Barnes}
\address{
  Department of Mathematics\\
  Louisiana State University\\
  Baton Rouge, LA 70803\\
  USA}
\email{mbarn31@math.lsu.edu}

\author{Bogdan Oporowski}
\address{
  Department of Mathematics\\
  Louisiana State University\\
  Baton Rouge, LA 70803\\
  USA}
\email{bogdan@math.lsu.edu}


\begin{document}

\begin{abstract}
We present a construction of two infinite graphs $G_1$ and $G_2$, and of an infinite set $\mathscr{F}$ of graphs such that $\mathscr{F}$ is an antichain with respect to the immersion relation and, for each graph $G$ in $\mathscr{F}$, both $G_1$ and $G_2$ are subgraphs of $G$, but no graph properly immersed in $G$ admits an immersion of $G_1$ and of $G_2$.  This shows that the class of infinite graphs ordered by the immersion relation does not have the finite intertwine property.

\end{abstract}

\date{\today}

\maketitle

\begin{section}{Introduction}

A \emph{graph} $G$ is a pair $(V(G),E(G))$ where $V(G)$, the set of vertices, is an arbitrary and possibly infinite set, and $E(G)$, the set of edges, is a subset of the set of two-element subsets of $V(G)$.
In particular, this definition implies that all graphs in this paper are simple, that is, with no loops or multiple edges.
The class of finite graphs will be denoted $\mathscr{G}_{<\infty}$ and the class of graphs whose vertex set is infinite will be denoted by $\mathscr{G}_\infty$.

Let $G$ and $H$ be graphs, and let $\mathscr{P}(G)$ denote the set of all nontrivial, finite paths of $G$.
We say $H$ is \emph{immersed} in $G$ if there is a map $\varphi : V(H) \cup E(H) \rightarrow V(G) \cup \mathscr{P}(G)$, sometimes abbreviated as $\varphi : H \rightarrow G$, such that:

\begin{enumerate}
  \item if $v\in V(H)$, then $\varphi (v)\in V(G)$;
  \item if $v$ and $v'$ are distinct vertices of $H$, then $\varphi(v) \neq \varphi (v') $;
  \item if $e = \{v,v'\} \in E(H)$, then $\varphi(e) \in \mathscr{P}(G)$ and the path $\varphi (e)$ connects $\varphi(v)$ with $\varphi(v')$;
  \item if $e$ and $e'$ are distinct edges of $H$, then the paths $\varphi(e)$ and $\varphi(e')$ are edge-disjoint; and
  \item if $e=\{v,v'\}\in E(H)$ and $v''$ is a vertex of $H$ other than $v$ and $v'$, then $\varphi (v'') \notin V(\varphi(e))$.
\end{enumerate}

We call $\varphi$ an \emph{immersion} and write $H \im G$.  It is easy to prove (see~\cite{siamd}) that the relation $\im$ is transitive.
If $C$ is a subgraph of $H$, then the restriction of $\varphi$ to $V(C)\cup E(C)$ will be abbreviated by $\varphi|_C$.
If $\varphi |_{V(H)}$ is a bijection such that two vertices, $v$ and $v'$, of $H$ are adjacent if and only if their images, $\varphi(v)$ and $\varphi(v')$, are adjacent in $G$, then we say that $\varphi$ induces an isomorphism between $H$ and $G$; otherwise $\varphi$ is \emph{proper}.
If $H = G$, then $\varphi$ is a \emph{self-immersion}, and, if additionally, it induces the identity map, then it is \emph{trivial}.
It is worth noting that immersion, as defined above, is sometimes called \emph{strong immersion}.

Let $S$ be a possibly infinite set of pairwise edge-disjoint paths in a graph $G$.
We say that $S$ is \emph{liftable} if no end-vertex of path in $S$ is an internal vertex of another path in $S$.
The operation of \emph{lifting} $S$ consists of deleting all internal vertices of all paths in $S$, and adding edges joining every pair of non-adjacent vertices of $G$ that are end-vertices of the same path in $S$.
It is easy to see that a graph $H$ is immersed in $G$ if and only if $H$ is isomorphic to a graph obtained from $G$ by deleting a set $V$ of vertices, deleting a set $E$ of edges, and then lifting a liftable set $S$ of paths.  
Furthermore, a self-immersion of $G$ is proper if and only if at least one of the sets $V$, $E$, and $S$ is nonempty.  

Given a graph $G$, a {\em blob} is a maximal 2-edge-connected subgraph of $G$.  Note that if a graph is 2-edge-connected, the graph itself is also a blob.  
An easy lemma about the immersion relation can be stated as follows.
\begin{lemma}
\label{immersionblob}
Let $H\im G$ via the immersion $\varphi$ and let $C$ be a blob of $H$.  Then there is a blob $D$ of $G$ such that $C\im D$ via the immersion $\varphi|_C$.
\end{lemma}

A pair $(\mathscr{G}, \leq)$, where $\mathscr{G}$ is a class of graphs and $\leq$ is a binary relation on $\mathscr{G}$, is called a \emph{quasi-order} if the relation $\leq$ is both reflexive and transitive.
A quasi-order $(\mathscr{G}, \leq)$ is a \emph{well-quasi-order} if it admits no infinite antichains and no infinite descending chains.

Suppose $(\mathscr{G}, \leq)$ is a quasi-order and $G_1$ and $G_2$ are two elements of $\mathscr{G}$.
An \emph{intertwine} of $G_1$ and $G_2$ is an element $G$ of $\mathscr{G}$ satisfying the following conditions:

\begin{itemize}
  \item $G_1 \leq G$ and $G_2 \leq G$, and
  \item if $G'\leq G$ and $G \nleq G'$, then $G_1 \nleq G$ or $G_2 \nleq G$.
\end{itemize}

The class of all intertwines of $G_1$ and $G_2$ is denoted by $\mathscr{I}_{\leq}(G_1,G_2)$.
A quasi-order $(\mathscr{G},\leq)$ satisfies the \emph{finite intertwine property} if for every pair $G_1$ and $G_2$ of elements of $\mathscr{G}$, the class of intertwines $\mathscr{I}_{\leq}(G_1, G_2)$ has no infinite antichains.
It is clear that if $(\mathscr{G},\leq)$ is a well-quasi-order, then it also satisfies the finite intertwine property.
However, it is well known that the converse is not true; for example, see~\cite{anoioig}.

Nash-Williams conjectured, and Robertson and Seymour later proved~\cite{gm23nwic} that
$(\mathscr{G}_{<\infty}, \im)$ is a well-quasi-order, and so it follows that
$(\mathscr{G}_{<\infty}, \im)$ satisfies the finite intertwine property.
In~\cite{anoioig}, the second author showed that $(\mathscr{G}_\infty , \m)$, where $\m$ denotes the minor relation, does not satisfy the finite intertwine property.
Andreae showed \cite{oioug} that $(\mathscr{G}_\infty , \im)$ is not a well-quasi-order.
In a result analogous to \cite{anoioig}, we strengthen Andreae's result by showing that $(\mathscr{G}_\infty , \im)$ does not satisfy the finite intertwine property.
In particular, we construct two graphs $G_1$ and $G_2$, and an infinite class $\mathscr{F}$ in $\mathscr{G}_\infty $ such that:

\begin{enumerate}
\item[(IT1)] $\mathscr{F}$ is an immersion antichain;
\item[(IT2)] every graph in $\mathscr{F}$ is connected;
\item[(IT3)] both $G_1$ and $G_2$ are subgraphs of each graph in $\mathscr{F}$;
\item[(IT4)] if $G'$ is properly immersed in a graph $G$ in $\mathscr{F}$, then $G_1 \nim G'$ or $G_2 \nim G'$.
\end{enumerate}

Note that~(IT3) implies that $G_1$ and $G_2$ are immersed in $G$.
Hence, the existence of graphs $G_1$, $G_2$ and a class of graphs $\mathscr{F}$ satisfying~(IT1)--(IT4) implies the following statement, which is the main result of the paper.

\begin{theorem}
\label{mainthm}
The quasi-order $(\mathscr{G}_{\infty}, \im)$ does not satisfy the finite intertwine property.
\end{theorem}

\end{section}
\begin{section}{The Construction}

We will exhibit two graphs $G_1$ and $G_2$ in $\mathscr{G}_{\infty}$ such that $\mathscr{I}_{\im}(G_1, G_2)$ is infinite.  The construction of $G_1$ and $G_2$ begins with the following results, which are immediate consequences of, respectively, Lemmas 3 and 4, and Lemmas 1 and 2 of \cite{osioig}.

\begin{theorem}
\label{thm1}
There is an infinite set $\mathscr{H}$ of pairwise-disjoint infinite blobs such that $|H| \leq |\mathscr{H}|$ for all $H\in \mathscr{H}$, and $\mathscr{H}$ forms an immersion antichain.
\end{theorem}

\begin{theorem}
\label{thm2}
Given an immersion antichain $\mathscr{H}$ of pairwise-disjoint infinite blobs such that $|H| \leq |\mathscr{H}|$ for all $H\in \mathscr{H}$, there is a connected graph $G$ such that the set of blobs of $G$ is $\mathscr{H}$ and $G$ admits no self-immersion except for the trivial one.
\end{theorem}

Let $\mathscr{H}$ be an antichain as described in Theorem~\ref{thm1}.
Partition $\mathscr{H}$ into countably many sets $\{\mathscr{H}_i\}_{i\in \Z}$ with the cardinality of each $\mathscr{H}_i$ equal to $|\mathscr{H}|$.
Then, by Theorem~\ref{thm2}, for each $i\in \Z$, there is a connected graph $B_i$ whose set of blobs is $\mathscr{H}_i$, and that admits no proper self-immersion.
Furthermore, Lemma \ref{immersionblob} implies that if $i$ and $j$ are distinct integers, then $B_i \nim B_j$, as no blob of $B_i$ is immersed in a blob of $B_j$.  Therefore, the set of graphs $\{B_i\}_{i\in \mathbb{Z}}$ is an immersion antichain.

For each graph $B_i$, label one vertex $u_i$.
Let $P$ be a two-way infinite path with vertices labeled $\{v_i\}_{i\in \Z}$ such that, for each integer $i$, the vertex $v_i$ is adjacent to $v_{i+1}$ and $v_{i-1}$.
We construct the graph $G_1$ by taking the disjoint union of $P$ and the graphs $B_i$ for which $i$ is odd, and then identifying the vertices $u_i$ and $v_j$ for $i=j$.
Similarly, we construct the graph $G_2$ by taking the disjoint union of $P$ and the graphs $B_i$ for which $i$ is even, and then identifying the vertices $u_i$ and $v_j$ for $i=j$.

Now let $j$ be an integer.  
Take the disjoint union of $G_1$ and all the graphs $B_i$ for which $i$ is even.  
Then, for each even integer $i$, identify the vertex $v_i$ of $G_1$ with the vertex $u_{i+2j}$ of the graph $B_{i+2j}$.  
Let $F_j$ be the resulting graph (see Figure~\ref{f1}) and define $\mathscr{F}$ as the set $\{F_j\}_{j\in \mathbb{Z}}$.

\end{section}

\begin{figure}[H]
\begin{center}
\begin{tikzpicture}[yscale=.8,xscale=1.2]
\draw[thick] (1.03,0) --(9,0);
\draw[thick][dotted] (.8,0) --(1,0);
\draw[thick][dotted] (9,0) --(9.3,0);
\draw[very thick, pattern=north west lines] (2,1) ellipse (.5 and 1);
\draw[very thick, pattern=north west lines] (4,1) ellipse (.5 and 1);
\draw[very thick, pattern=north west lines] (6,1) ellipse (.5 and 1);
\draw[very thick, pattern=north west lines] (8,1) ellipse (.5 and 1);
\node [above] at (2,2) {$B_i$};
\node [above] at (6,2) {$B_{i+2}$};
\node [above] at (4,2) {$B_{i+1+2j}$};
\node [above] at (8,2) {$B_{i+3+2j}$};
\node [below] at (2,-.20) {$v_i$};
\node [below] at (6,-.20) {$v_{i+2}$};
\node [below] at (4,-.20) {$v_{i+1}$};
\node [below] at (8,-.20) {$v_{i+3}$};
\node [left] at (.5,.25) {$P$};
\draw[fill] (2,0) circle [radius=0.1];
\draw[fill] (4,0) circle [radius=0.1];
\draw[fill] (6,0) circle [radius=0.1];
\draw[fill] (8,0) circle [radius=0.1];
\end{tikzpicture}
\caption{The graph $F_j$}
\label{f1}
\end{center}
\end{figure}
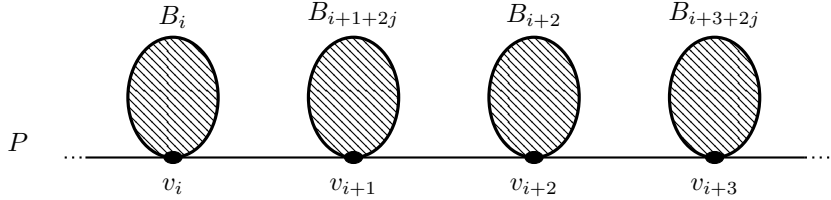

The following lemma immediately implies our main result, Theorem~\ref{mainthm}.

\begin{lemma}
The set of graphs $\mathscr{F}=\{F_j\}_{j\in \mathbb{Z}}$ is an immersion antichain.  Furthermore, each $F_j \in \mathscr{F}$ is an immersion intertwine of the graphs $G_1$ and $G_2$.
\end{lemma}

\begin{proof}
Let $j$ be an integer.  
It is easy to see that $F_j$ satisfies (IT2) and (IT3).  
Therefore, in order to show that $F_j$ is an immersion intertwine of $G_1$ and $G_2$, it suffices to prove that it also satisfies (IT4).

Suppose, for contradiction, that $F'_j$ is a graph that is properly immersed in $F_j$ via a map $\varphi$, and both $G_1$ and $G_2$ are immersed in $F'_j$.
Then we can obtain $F'_j$ from $F_j$ by deleting a set of vertices $V$, deleting a set of edges $E$, and then lifting a liftable set of paths $S$, with at least one of these sets being nonempty.  We consider two cases depending on whether there is an integer $i$ for which $B_i$ meets $V\cup E \cup S$.

First, assume that no $B_i$ meets $V\cup E \cup S$.  
Then the sets $V$ and $S$ are empty, as all the vertices of $F_j$ are contained in the subgraphs $\{B_n\}_{n\in  \Z}$, and $E$ consists of some edges of $P$.

Suppose the edge~$e=\{v_k, v_{k+1}\}$ is in $E$ where $k$ is odd; the argument is symmetric when $k$ is even.
The graph $F_j \setminus e$ has exactly two components, with the subgraphs $B_k$ and $B_{k+2}$ in distinct components.
Label the component containing $B_k$ as $C_1$ and the component containing $B_{k+2}$ as $C_2$.

Let $A$ be a blob of $B_k$.
As $A$ and each blob of $C_2$ are members of the antichain $\mathscr{A}$, by Lemma~\ref{immersionblob}, we have $A\nim C_2$.  
Hence, by transitivity, $B_k\nim C_2$.  It follows similarly that $B_{k+2}\nim C_1$.
But as $G_1$ is connected and the only components of $F_j\setminus e$ are $C_1$ and $C_2$, we have that $G_1 \nim F_j \setminus e$.
Furthermore, as $F'_j \im F_j \setminus e$, by transitivity, $G_1 \nim F'_j$; a contradiction.

Now suppose that, for some odd integer $i$, the graph $B_i$ meets $V\cup E \cup S$; again, the argument is symmetric if $i$ is even.
As $G_1$ is immersed in $F'_j$, so is $B_i$.  Let $T$ be the subgraph of $F'_j$ induced by $\varphi^{-1} (V(B_i)\cup \mathscr{P}(B_i))$, and let $\psi$ be the immersion of $B_i$ into $F'_j$.
As $B_i$ admits no proper self-immersion, there must be some vertex $v$ of $B_i$ such that $\psi(v)$ is a vertex of $F'_j - T$.

Let $A_v$ be the blob of $B_i$ containing $v$.
By Lemma~\ref{immersionblob}, the blob $A_v$ is immersed in some blob of $F'_j - T$.  But, again by Lemma~\ref{immersionblob}, each blob of $F'_j - T$ is immersed in a graph of the antichain $\mathscr{A} \setminus \{A_v\}$.  So $A_v$ cannot be immersed in $F'_j - T$.  Therefore, $B_i$ is not immersed in $F'_j$ and neither is $G_1$.

Hence, $\mathscr{F}$ satisfies the condition (IT4).

To show that $\mathscr{F}$ is an antichain in $(\mathscr{G}_\infty, \im)$, suppose that $F_i$ is immersed in $F_j$ for some distinct integers $i$ and $j$. 
By construction, $F_i$ and $F_j$ are not isomorphic.  
Therefore, $F_i$ is properly immersed in the intertwine $F_j$ and so either $G_1 \nim F_i$ or $G_2 \nim F_i$.  
But both $G_1$ and $G_2$ are immersed in $F_i$ by construction; a contradiction.  The conclusion follows.
\end{proof}

The graphs $\{B_i\}_{i\in \Z}$ used in our construction, whose existence was proved in~\cite{oioug}, have vertex sets of very large cardinality.  In fact, the cardinal in question is the first limit cardinal greater than the cardinality of the continuum.  It is not known whether the class of graphs of smaller cardinality ordered by the strong immersion relation is a well-quasi-ordering, whether it has the finite intertwine property, and whether there exists a infinite graph of smaller cardinality that admits only the trivial self-immersion.

\end{document}